\documentclass{svjour3}
\usepackage[utf8]{inputenc}
\usepackage{amsmath,amssymb,amsfonts}
\usepackage{comment,mathdots,url}
\usepackage{xcolor}

\usepackage{graphicx}

\newcommand{\iu}{\underline i}
\usepackage{listings}
\usepackage{algorithmic}
\usepackage{algorithm}
\usepackage{color}

\definecolor{MyLightGrey}{rgb}{0.95,0.95,0.95}
\title{On certain matrix algebras related to quasi-Toeplitz matrices
}
\author{Dario A.~Bini \and Beatrice Meini }

\institute{D.~A.~Bini\at
              Universit\`a di Pisa, Italy \\
     \email{dario.bini@unipi.it}
           \and
           B.~Meini\at
           Universit\`a di Pisa, Italy \\
     \email{beatrice.meini@unipi.it} \\
     Corresponding author
}

\begin{document}
\maketitle

\begin{flushleft}
 \textit{Dedicated to our friend and colleague Marc Van Barel}    
\end{flushleft}

 \abstract{Let $A_\alpha$ be the semi-infinite tridiagonal matrix having subdiagonal and superdiagonal unit entries, $(A_\alpha)_{11}=\alpha$,  where  $\alpha\in\mathbb C$, and zero elsewhere. A basis $\{P_0,P_1,P_2,\ldots\}$ of the linear space $\mathcal P_\alpha$ spanned by the powers of $A_\alpha$ is determined, where $P_0=I$, 
 $P_n=T_n+H_n$,  $T_n$ is the symmetric Toeplitz matrix having ones in the $n$th super- and sub-diagonal, zeros elsewhere, and $H_n$ is the Hankel matrix with
 first row $[\theta\alpha^{n-2}, \theta\alpha^{n-3},  \ldots, \theta, \alpha, 0, \ldots]$, where $\theta=\alpha^2-1$.
The set $\mathcal P_\alpha$ is an algebra, and for
$\alpha\in\{-1,0,1\}$, $H_n$ has only one nonzero anti-diagonal. This fact is exploited to provide a better representation of symmetric quasi-Toeplitz matrices $\mathcal {QT}_S$,  where, instead of
representing a generic matrix $A\in\mathcal{QT}_S$ as $A=T+K$, where $T$ is Toeplitz and 
$K$ is compact, it is represented as $A=P+H$, 
where $P\in\mathcal P_\alpha$
and $H$ is compact. 
It is shown experimentally that the matrix arithmetic obtained this way is much more effective than that 
implemented in the toolbox {\tt CQT-Toolbox} of {\em Numer.~Algo.} 81(2):741--769, 2019.
}

\section{Introduction}
In this paper, we analyze structural and computational properties of the linear space $\mathcal P_\alpha$ spanned by the powers of the semi-infinite tridiagonal matrix $A_\alpha=(a_{i,j})_{i,j\in\mathbb Z^+}$, where $a_{i,i+1}=a_{i+1,i}=1$, $a_{1,1}=\alpha$ and $a_{i,j}=0$ elsewhere, i.e.,
\[
A_\alpha=\begin{bmatrix}
\alpha&1\\
1&0&1\\
 &\ddots&\ddots&\ddots
\end{bmatrix}.
\]

More specifically, let us denote $a(z)=\sum_{i\in\mathbb Z}a_iz^i$ the formal Laurent power series having coefficients $a_i\in\mathbb C$ and let $T(a)$ be the associated Toeplitz matrix defined by $t_{i,j}=a_{j-i}$. Recall that $T(a)$ is symmetric if and only if $a_i=a_{-i}$ for $i\in\mathbb Z^+$.
We show that { a semi-infinite matrix} $A$ is in $\mathcal P_\alpha$ if and only if there exists a vector $a=(a_i)_{i\ge 0}$ such that, for $a(z)=a_0+\sum_{i=1}^\infty a_i(z+z^{-1})$, one has
$A=T(a)+H_\alpha(a)$, where  $H_\alpha=(h_{i,j})$, $h_{i,j}=\eta_{i+j-1}$ is the Hankel matrix whose first column $\eta=(\eta_i)$ is such that
\[
\eta=\begin{bmatrix}
v_2&v_3&v_4&v_5&\ldots&\ldots&\ldots\\
v_1&v_2&v_3&v_4&v_5&\ddots&\ddots\\
   &v_1&v_2&v_3&v_4&v_5&\ddots\\
   &   &\ddots&\ddots&\ddots&\ddots&\ddots
\end{bmatrix}
\begin{bmatrix}
a_1\\ a_2\\ a_3 \\ \vdots
\end{bmatrix},
\]
where $v_1=1$, $v_2=\alpha$, $v_{3+i}=\theta \alpha^i$, $i=0,1,\ldots$, and $\theta=\alpha^2-1$.
{ Here, we assume conditions on $\alpha$ and $a_i$ under which the components $\eta_i$ have finite values. }
Observe that the above matrix is a Toeplitz matrix in Hessenberg form. Moreover, for $\alpha=0$ the matrix is tridiagonal skew-symmetric, while, if $\alpha=\pm 1$ the matrix is bidiagonal and lower triangular.

Moreover we show that if $|\alpha|\le 1$ and the coefficients $a_i$ are such that
$\sum_{i=1}^\infty |a_i|<\infty$,
then the matrices in $\mathcal P_\alpha$ have  2-norm and infinity norm bounded from above, so that  $\mathcal P_\alpha$ is a Banach algebra if endowed with $\|\cdot\|_2$ or with $\|\cdot\|_\infty$.

If $|\alpha|>1$ and the coefficients $a_i$ are restricted to values such that the function $a(z)=a_0+\sum_{i=1}^\infty a_i(z^i+z^{-i})$ is analytic in the annulus $\mathcal A(r)= \{ z\in\mathbb C~ | \quad r^{-1}<|z|<r \}$ 
for some $r>|\alpha|$, then $\mathcal P_\alpha$ is still a Banach algebra.

The above results are rephrased in the case where the matrix $A_\alpha$ is truncated to size $n$ by setting $a_{n,n}=\beta$ for $\beta\in\mathbb C$. Let us denote by $A_{\alpha,\beta}$ the matrix obtained this way so that
\[
A_{\alpha,\beta}=\begin{bmatrix}
\alpha&  1  \\
1     &  0   &1\\
      &\ddots&\ddots &\ddots\\
      &      &1      &0     &1\\
      &      &       &1     &\beta
\end{bmatrix}.
\] 
In this case, the span $\mathcal P_{\alpha,\beta}$ of the powers $A_{\alpha,\beta}^i$, $i=0,1,\ldots,n-1$, is a matrix algebra formed by matrices of the kind $T(a)+H_\alpha(a)+JH_\beta(a)J$, where $J$ is the flip matrix having ones on the antidiagonal and zeros elsewhere.
This algebra, for $\alpha,\beta\in\{-1,0,1\}$ includes the algebras investigated in \cite{bini-capovani,bozzo-phd,bozzo-difiore,ekstrom,olshevsky}. 

As an application of these results, we provide a better representation of quasi-Toeplitz matrices having a symmetric Toeplitz part. We recall that
a semi-infinite matrix $A=(a_{i,j})_{i,j\in\mathbb Z^+}$ is quasi-Toeplitz (QT) if it can be written as $A=T(a)+E_A$ where $T(a)$ is the Toeplitz part and $E_A$ is a compact linear operator from $\ell^2$ in $\ell^2$, with $\ell^2=\{x=(x_i)_{i\in\mathbb Z^+}~|\quad ~\sum_{i\in\mathbb Z^+}|x_i|^2<\infty\}$
\cite{bmm-mathcomp}.

Quasi-Toeplitz matrices are encountered in queueing and network models, like the Tandem Jackson Network \cite{Motyer-Taylor}, that can be described by random walks in the quarter plane \cite{bmm-mathcomp,bmmr:scicomp}. In \cite{bmr:cqt} a matrix arithmetic for QT matrices has been designed and analyzed. This software, implemented as the Matlab toolbox {\tt CQT-toolbox}, has been used in \cite{bmm:simax} to solve certain quadratic matrix equations with QT coefficients encountered in the analysis of QBD processes \cite{blm:book}, \cite{lr:book}.

Here we show that representing a QT matrix $A=T(a)+E_A$, where $T(a)$ is symmetric, in the form $P_\alpha(a)+K_A$, for $K_A$ compact, is more convenient from the theoretical and the computational point of view.

In fact, since the set $\mathcal P_\alpha$ is an algebra, we find that $f(P_\alpha(a))=P_\alpha(f(a))$, for any rational function $f(z)$ defined in the set $\{a(z)~|\quad |z|=1\}$. This fact allows to reduce the computation of the matrix function $f(P_\alpha(a))$ to the composition of two functions. This property is not satisfied by the matrices $T(a)$ since symmetric Toeplitz matrices do not form a matrix algebra.

From the computational point of view, we compared the representation of QT matrices given in the {\tt CQT-Toolbox}, designed in \cite{bmr:cqt}, with the new representation based on the algebra $\mathcal P_\alpha$. As test problems, we considered the solution of a matrix equation with QT coefficients performed by means of fixed point iterations, and the computation of the square root of a real symmetric QT matrix. In both cases,
 with the new representation we obtain substantial acceleration in the CPU time with respect to the standard representation given in \cite{bmr:cqt}.

The paper is organized as follows. In the next Section 2, we briefly recall some preliminary notions needed in our analysis. In Section 3, { we introduce the linear space spanned by the powers of $A_\alpha$ and analyze its structural properties}. In Section 4 we discuss some boundedness issues and give conditions in order that $\mathcal P_\alpha$ is a Banach algebra. The cases where $\alpha\in\{-1,0,1\}$ are presented in Section 5. In Section 6 we show an application of these results to provide a better representation of QT matrices. Section 7 reports the results of some numerical experiments and finally Section 8 draws the conclusions. 

\section{Preliminaries}
Denote $\mathbb T=\{z\in\mathbb C~|\quad |z|=1\}$ the unit circle in the complex plane, $\mathbb Z$ the set of relative integers, $\mathbb Z^+=\{i\in\mathbb Z~|\quad i>0\}$ the set of positive integers.
Let $\mathcal W=\{f(z)=\sum_{i\in\mathbb Z}f_iz^i:\mathbb T\to\mathbb C~|\quad \|f\|_W:=\sum_{i\in\mathbb Z}|f_i|<+\infty\}$
be the Wiener algebra. 
 Given $a(z)=\sum_{i\in\mathbb Z}a_iz^i\in\mathcal W$, let $T(a)=(t_{i,j})_{i,j\in\mathbb Z^+}$ be the semi-infinite Toeplitz matrix with entries 
$t_{i,j}=a_{j-i}\in\mathbb C$, for $i,j\in\mathbb Z^+$.

We say that $A$ is a quasi-Toeplitz matrix (in short QT-matrix) if $A=T(a)+E_A$, where  $E_A$ 
represents a compact linear operator from $\ell^2$ into $\ell^2$, with $\ell^2=\{x=(x_i)_{i\in\mathbb Z^+}~|\quad \sum_{i=1}^\infty |x_i|^2<\infty\}$.
We refer to $T(a)$ as the {\em Toeplitz part} of $A$ and to $E_A$ as the {\em compact part} of $A$, and we write $E_A=\kappa(A)$. The set of QT matrices is denoted by $\mathcal {QT}$.

 Denote by $\mathcal W_S\subset\mathcal W$ the set of functions $a(z)=\sum_{i\in\mathbb Z}a_iz^i\in\mathcal W$ having symmetric coefficients, i.e., such that $a_i=a_{-i}$, and by $\mathcal{QT}_S$ the set  of QT matrices of the kind $A=T(a)+E_A$ where $a\in\mathcal W_S$.

Given QT matrices $A=T(a)+E_A$, $B=T(b)+E_B$, we 
write $A\doteq B$ if $a(z)=b(z)$, that is, if $A-B$ is a compact operator.

Let $v(z)=\sum_{i=1}^\infty v_{i}z^i\in\mathcal W$ and denote with the same symbol $v$ the semi-infinite column vector $v=(v_1,v_2,\ldots)^T$. Define $H(v)=(v_{i+j-1})_{i,j\in\mathbb Z^+}$ the semi-infinite Hankel matrix associated with $v(z)$.
Observe that the first column of $H(v)$ coincides with the vector $v$.

Given $a(z)=\sum_{i\in\mathbb Z}a_iz^i\in\mathcal W$, define the following power series in the Wiener class: $a_+(z)=\sum_{i=1}^{\infty}a_iz^{i}$ and 
$a_-(z)=\sum_{i=1}^{\infty}a_{-i}z^{i}$.

We { recall the following classical result, see for instance \cite{BG:book2005}}.

\begin{theorem}\label{thm:bg}
For any $a(z)\in\mathcal W$,  the matrices 
$H(a_+)$ and $H(a_-)$ are compact operators.  Moreover, for $b(z)\in\mathcal W$ it holds that
\[
T(a)T(b)=T(ab)-H(a_-)H(b_+),
\]
where $H(a_-)H(b_+)$ is compact.
\end{theorem}

The above result implies that
 for any QT matrices $A=T(a)+E_A$, $B=T(b)+E_B$ it holds  $AB\doteq T(ab)$. As a consequence we have
\[
(T(a)+E_A)^n\doteq T(a^n).
\]

In the following, we restrict our attention to the subset $\mathcal{QT_S}$.

\section{A family of matrix algebras}\label{sec:basic}

For a given $\alpha\in\mathbb C$, denote 
\begin{equation}\label{eq:aalpha}
    A_\alpha=T(z+z^{-1})+\alpha e_1e_1^T=
    \begin{bmatrix}
    \alpha&1\\
    1&0&1\\
     &\ddots&\ddots&\ddots
    \end{bmatrix},
\end{equation}
where $e_1$ is the first column of the semi-infinite identity matrix $I$, and blanks denote null entries.
The goal of this section is to provide a structural analysis of the matrices belonging to the linear space spanned by the powers $A_\alpha^i$, $i=0,1,\ldots,n$ { for any $n\in\mathbb Z^+$, that is,
\[
\mathcal S_{n,\alpha}=\left\{\sum_{i=0}^ns_iA_\alpha^i\quad |\quad \quad s_i\in\mathbb R\right\},\quad n\in\mathbb Z^+.
\]
} In this regard we will provide a different basis of this space that allows one to represent each matrix in the space as the sum of a banded Toeplitz and a Hankel matrix where only the Hankel component depends on $\alpha$.

More specifically,
we will introduce a new set $\{P_{0,\alpha},P_{1,\alpha},\ldots,P_{n,\alpha}\}$, of symmetric QT matrices having a better representation than $A_\alpha^i$, and prove that { this set is a basis of the linear space $\mathcal S_{n,\alpha}$.}

Set $\theta=\alpha^2-1$,  and define the polynomials in the variable $z$
\begin{equation}\label{eq:v}
\begin{split}
& h_{1}(z)=\alpha z,  \\
& h_{n}(z)=\theta\sum_{i=1}^{n-1}\alpha^{n-i-1}z^i+\alpha z^{n},~~ n\ge 2.
\end{split}
\end{equation}
Denote with the same symbol $h_{n}$ the infinite column vector formed by the coefficients of the term $z^i$, for $i=1,\ldots,n$, in the polynomial $h_{n}(z)$, filled with zeros. Denote also
\begin{equation}\label{eq:hn}
    H_{n,\alpha}=H(h_{n})
\end{equation}
the semi-infinite Hankel matrix associated with $h_{n}$ whose first column is the vector $h_n$. Finally,
define \begin{equation}\label{eq:pn}
\begin{split}
&P_{0,\alpha}=I,\\
&P_{n,\alpha}=T(z^n+z^{-n})+H_{n,\alpha}, \quad \hbox{for }n\ge 1,
\end{split}
\end{equation}
and observe that $P_{1,\alpha}=A_\alpha$.
For notational simplicity, if the dependence on $\alpha$ is clear from the context,
 we will write  $P_n$ in place of $P_{n,\alpha}$, and $H_n$ in place of $H_{n,\alpha}$.

In particular, for $n=4$ we have
\[
h_{4}=(\theta\alpha^2,\theta\alpha,\theta,\alpha,0,\ldots)^T
\]
and
\[
P_4=\begin{bmatrix}
\theta\alpha^2&\theta\alpha&\theta&\alpha&1&\\
\theta\alpha &\theta&\alpha&&&1\\
\theta&\alpha&&&&&1\\
\alpha&      &&&&&&1\\
1     &      &&&&&&&\ddots\\
&1\\
&&\ddots
\end{bmatrix},
\]
 where blank entries stand for zeros.

 We need { to prove a few preliminary results  to show that the  set $\{P_{i,\alpha}:~i=0,1,\ldots,n\}$ is a basis of the linear space $\mathcal S_{n,\alpha}$.}
The first result that we need is the following identity, that is a direct consequence of the binomial expansion of $(z+z^{-1})^n$:
\begin{equation}\label{eq:id}
\begin{split}
    &(z+z^{-1})^n=\sum_{i=0}^{\lfloor\frac {n-1}2\rfloor}
    \binom{n}{i}
    (z^{n-2i}+z^{-n+2i})+\varphi_n\\
&\varphi_n=0 \hbox{ if }n\hbox{ is odd, } \quad \varphi_n=
\binom{n}{\frac n2}
 \hbox{ if }n \hbox{ is even}.
\end{split}
\end{equation}

Another useful property concerns the compact part $K_n:=\kappa(A_\alpha^n)$ of the QT matrix $A_\alpha^n$, i.e.,  $K_n=A_\alpha^n-T((z+z^{-1})^n)$.

\begin{lemma}\label{lem1}
The compact part $K_n$ of $A_\alpha^n$ is a Hankel matrix. Moreover, its first column $k_n=K_ne_1$ is such that
\begin{equation}\label{eq:lem1}
k_{n+1}=A_\alpha k_n+\sigma_ne_1,~~n\ge 1,
\end{equation}
where $\sigma_n=\alpha\binom{n}{n/2}$
 if $n$ is even, $\sigma_n=-\binom{n}{\lfloor n/2\rfloor}$
if $n$ is odd.
\end{lemma}
\begin{proof}
Since $K_n=A_\alpha^n-T((z+z^{-1})^n)$ and both addends are symmetric, then $K_n$ is symmetric. 
We prove that $K_n$ is Hankel by induction on $n$. Clearly, the property is true for $n=1$ where $K_n=\alpha e_1e_1^T$. Assume by inductive assumption that $K_n$ is Hankel. Then,
since $K_{n+1}=\kappa(A_\alpha  A_\alpha^n)$, a direct inspection  shows that
\[
K_{n+1}=A_\alpha\kappa(A_\alpha^n)+\kappa(A_\alpha)T(g^n)+\kappa(T(g)T(g^n)),
\]
where $g(z)=z+z^{-1}$. Since $\kappa(A_\alpha)=\alpha e_1e_1^T$, $\kappa(A_\alpha^n)=K_n$, and, from Theorem \ref{thm:bg}, $\kappa(T(g)T(g^n))=e_1e_1^T H(g^n)$, we find that
\begin{equation}\label{eq:lem1tmp}
K_{n+1}=A_\alpha K_n+\alpha e_1e_1^TT(g^n)+e_1e_1^TH(g^n).
\end{equation}
Removing the first row of $K_{n+1}$ we are left with the submatrix of 
$A_\alpha K_n$ 
obtained by removing its first row. This coincides with the product
\[
\begin{bmatrix}
1& 0 &1\\
 &1& 0 &1\\
 & &\ddots&\ddots&\ddots
\end{bmatrix}K_n,
\]
that is, the sum of $K_n$ plus the matrix obtained by shifting up the rows of $K_n$ by two places. Since both the addends are Hankel matrices by the inductive assumption, then their sum is Hankel as well. Now, since $K_{n+1}$ is symmetric and its submatrix obtained removing the first row is Hankel, then the whole matrix $K_{n+1}$ is Hankel.

In order to prove \eqref{eq:lem1}, it is sufficient to multiply both sides of \eqref{eq:lem1tmp} to the right by $e_1$. This manipulation yields
\[
k_{n+1}=A_\alpha k_n+(\alpha e_1^TT(g^n)e_1)e_1+(e_1^TH(g^n)e_1)e_1.
\]
To complete the proof, it is sufficient to determine the value of $\sigma_n:= \alpha e_1^TT(g^n)e_1+e_1^TH(g^n)e_1$. Now, $\gamma_n:=e_1^TT(g^n)e_1$ is the constant term of the Laurent polynomial $g(z)^n$. In view of \eqref{eq:id}, we have $\gamma_n=0$ if $n$ is odd and $\gamma_ n = 
\binom{n}{n/2}$ if $n$ is even. On the other hand, $\delta_n:=e_1^TH(g^n)e_1$ is the coefficient of $z+z^{-1}$ in the Laurent polynomial $g(z)^n$, and applying once again \eqref{eq:id} we get $\delta_n=
\binom{n}{\lfloor n/2\rfloor}$
if $n$ is odd, while $\delta_n=0$ if $n$ is even. This proves \eqref{eq:lem1} with $\sigma_n=\alpha\gamma_n-\delta_n$ that completes the proof. \qed
\end{proof}

Let $h_n(z)$ be the function defined in \eqref{eq:v}, consider the matrix $H_n=H(h_n)$  which is the compact part of $P_n=T(z^n+z^{-n})+H_n$,
and recall that $h_n=H_ne_1$ is the first column of $H_n$.

The following result relates the vectors $h_n$ and the matrix $A_\alpha$.

\begin{lemma} \label{lem2}
Let $h_{n}(z)$ be the function defined in \eqref{eq:v}, $h_n$ the associated coefficient vector, and let $A_\alpha=T(z+z^{-1})+\alpha e_1e_1^T$.  Then we have
\begin{equation}\label{eq:lem2}
A_\alpha h_n=h_{n-1}+h_{n+1},\quad n\ge 1,
\end{equation}
where we have set $h_0:=e_1$.
\end{lemma}
\begin{proof}
The expressions for $A_\alpha h_1$ and $A_\alpha h_2$ follow by a direct inspection.
For $n\ge 3$,
consider the $i$th component in both sides of equation \eqref{eq:lem2}. For $i=1$, we have 
$(A_\alpha h_n)_1=\theta \alpha^{n-1}+\theta\alpha^{n-3}$, while $(h_{n+1})_1=\theta\alpha^{n-1}$, $(h_{n-1})_1=\theta\alpha^{n-3}$.
This shows that \eqref{eq:lem2} is satisfied in the first component. Concerning the case $i>1$, we  may equivalently rewrite the condition 
$(A_\alpha h_n)_i=(h_{n-1})_i+(h_{n+1})_i$, $i>1$,
as $ZA_\alpha h_n=Zh_{n-1}+Zh_{n+1}$ where 
\[
Z:=T(z)=
\begin{bmatrix}
0&1\\
 &0&1\\
 & &\ddots&\ddots
\end{bmatrix}
\]
is the semi-infinite shift matrix, and get $(I+Z^2)h_n=Zh_{n-1}+Zh_{n+1}$ since $ZA_\alpha=I+Z^2$. This equation is satisfied since $Zh_{n+1}=h_n$ and $Zh_{n-1}=h_{n-2}=Z^2h_n$.
\qed
\end{proof}

Now, we are ready to prove the main result of this section.

\begin{theorem}\label{thm:main}
Let $K_n=\kappa(A_\alpha^n)$ and $H_n=H(h_{n})$, where  $h_{n}(z)$ is the function defined in \eqref{eq:v}.
For the vectors $h_n=H_ne_1$ and $k_n=K_ne_1$
we have
\[
k_n=\sum_{i=0}^{\lfloor\frac {n-1}2\rfloor}
\binom{n}{i}h_{n-2i},\quad n=1,2,\ldots .
\]
\end{theorem}
\begin{proof}
We prove this identity by induction on $n$. For $n=1$ we have $k_1=\alpha e_1=h_1$. For the inductive step, assume that the identity holds for $n$ and prove it for $n+1$. 
By Lemma \ref{lem1} and by the inductive assumption, we have
\begin{equation}\label{eq:p1}
k_{n+1}=A_\alpha k_n+\sigma_n e_1=A_\alpha\sum_{i=0}^{\lfloor\frac {n-1}2\rfloor}\binom{n}{i}
h_{n-2i}+\sigma_n e_1.
\end{equation}
If $n$ is odd, applying Lemma \ref{lem2} yields
\[\begin{split}
k_{n+1}=&\sum_{i=0}^{\frac {n-1}2}\binom{n}{i}
(h_{n-2i-1}+h_{n-2i+1})+\sigma_n e_1\\
=&h_{n+1}+\sum_{i=1}^{\frac {n-1}2}\left(
\binom{n}{i}+\binom{n}{i-1}
\right)
h_{n-2i+1} +\binom{n}{\frac{n-1}2}
h_0+\sigma_ne_1.
\end{split}\]
Now, recall that $h_0=e_1$, and $\sigma_n=
-\binom{n}{frac{n-1}2}$ so that
$\binom{n}{\frac{n-1}2} h_0+\sigma_ne_1=0$, moreover,
\begin{equation}\label{eq:p2}
\binom{n}{i}+\binom{n}{i-1}
=\binom{n+1}{i},\quad i\ge 1,
\end{equation}
so that we
get
\[
k_{n+1}=h_{n+1}+\sum_{i=1}^{\frac {n-1}2}
\binom{n+1}{i} 
h_{n-2i+1}=\sum_{i=1}^{\frac {n+1}2}
\binom{n+1}{i}
h_{n+1-2i},
\]
which completes the inductive step for $n$ odd.

Now consider the case where $n$ is even.
Then applying Lemma \ref{lem2} to \eqref{eq:p1} yields 
\[\begin{split}
k_{n+1}=&\sum_{i=0}^{{\frac n2}-1}
\binom{n}{i}
(h_{n-2i-1}+h_{n-2i+1})+
\binom{n}{\frac n2-1}
h_1+\sigma_n e_1\\
=&h_{n+1}+\sum_{i=1}^{{\frac n2}-1}\left(
\binom{n}{i}+\binom{n}{i-1}
\right)
h_{n-2i+1} +
\binom{n}{\frac n2-1}
h_1 +\sigma_n e_1.
\end{split}\]
Since $h_1=\alpha e_1$ and $\sigma_n=\alpha\binom{n}{\frac n2}
$, applying \eqref{eq:p2} yields
\[
k_{n+1}=\sum_{i=0}^{\frac {n}2}
\binom{n+1}{i}
h_{n+1-2i},
\]
that is the desired equation since, for $n$ even, $\frac n2=\lfloor\frac {n+1} 2\rfloor$.\qed
\end{proof}

An immediate corollary of Theorem \ref{thm:main} is given below

\begin{corollary} For the matrices $A_\alpha^n$, $n=1,2,\ldots$, we have 
\begin{equation}\label{eq:tmp}
A_\alpha^n=\sum_{i=0}^{\lfloor\frac {n-1}2\rfloor}
\binom{n}{i}
P_{n-2i,\alpha}+\varphi_n I,
\end{equation}
where $P_{i,\alpha}$, $i=0,1,\ldots,n$, are the matrices defined in \eqref{eq:pn} and $\varphi_n$ is defined in \eqref{eq:id}.
\end{corollary}
\begin{proof}
We separately show that the Toeplitz part and the compact part in both sides of \eqref{eq:tmp} coincide. Concerning the Toeplitz part, 
it is sufficient to prove the equality of the symbols, that is,
\[
(z+z^{-1})^n=\sum_{i=0}^{\lfloor\frac {n-1}2\rfloor}
\binom{n}{i}
(z^{n-2i}+z^{-n+2i})+\varphi_n,
\]
which coincides with equation \eqref{eq:id}.
Concerning the compact part, since $\kappa(A^n)$ and $\kappa(P_{n-2i})$ are Hankel matrices defined by their first column $k_n$ and $h_{n-2i}$, respectively, we may equivalently rewrite the condition
\[
\kappa(A_\alpha^n)=\sum_{i=0}^{\lfloor\frac {n-1}2\rfloor}
\binom{n}{i}
\kappa(P_{n-2i})
\]
as  
\[
k_n=\sum_{i=0}^{\lfloor\frac {n-1}2\rfloor}
\binom{n}{i}
h_{n-2i}.
\]
The latter vector equation follows directly from Theorem \ref{thm:main}.\qed
\end{proof}

The following result shows that equation \eqref{eq:tmp} allows to express $P_{n,\alpha}$ as a linear combination of $A_\alpha^i$, for $i=0,\ldots,n$.

\begin{corollary}
The linear space spanned by the matrices $P_{i,\alpha}$, $i=0,1,\ldots,n$, defined in \eqref{eq:pn}, coincides with $\mathcal S_{n,\alpha}$.
\end{corollary}
\begin{proof}
Clearly, equation \eqref{eq:tmp} provides an explicit expression of $A_\alpha^n$ in terms of the matrices $P_{i,\alpha}$, $i=0,1,\ldots,n$. Conversely, we may write \eqref{eq:tmp} as $P_{n,\alpha}= A_\alpha^n -\sum_{i=1}^{\lfloor \frac {n-1}2\rfloor}
\binom{n}{i}
P_{n-2i,\alpha}+\varphi_n I$. By means of an inductive argument it follows that any matrix $P_{j,\alpha}$, for $j=0,1,\ldots,n$, can be expressed as a linear combination of the matrices $A_\alpha^i$ for $i=0,1,\ldots, j$. This completes the proof. \qed
\end{proof}

Given a symmetric Laurent polynomial $a(z)=a_0+\sum_{i=1}^n a_i(z^i+z^{-i})$, denote $P_\alpha(a)\in\mathcal S_{n,\alpha}$ the matrix
\[
P_\alpha(a)=\sum_{i=0}^n a_iP_{i,\alpha}.
\]
We say that $P_\alpha(a)$ is the QT matrix in $\mathcal S_{n,\alpha}$ associated with the symbol $a(z)$. As a synthesis of the results of this section we may conclude that
\begin{equation}\label{eq:rep}
P_\alpha(a)=T(a)+H_\alpha(a),\quad H_\alpha(a)=\sum_{i=1}^n a_iH_{i,\alpha}.
\end{equation}
That is, the symbol $a(z)$ uniquely defines the Toeplitz part $T(a)$ and the compact (Hankel) part $H_\alpha(a)$ of $P_\alpha(a)$. In particular, any matrix in $\mathcal S_{n,\alpha}$ is the sum of a symmetric Toeplitz matrix and a Hankel matrix associated with the same symbol $a(z)$.

One may wonder if the matrix $P_\alpha(a)$ can be defined in the case where $a(z)=a_0+\sum_{i=1}^\infty a_i(z^i+z^{-i})$ is a Laurent series in $\mathcal{W_S}$.
We provide an answer to this question in the next section.

\section{Boundedness issues}
A natural issue concerns the boundedness of 
\begin{equation}
\label{eq:palphainf}
P_\alpha(a)=\sum_{i=0}^\infty a_i P_{i,\alpha}    
\end{equation}
in the case where
$a(z)=a_0+\sum_{i=1}^\infty a_i(z^i+z^{-i})$ is a power series such that $\sum_{i=0}^\infty |a_i|<\infty$, that is, $a(z)\in\mathcal W_S$. In order to answer this question,
observe that { from \eqref{eq:v} we have}
$\|h_n\|_1=|\alpha|+|\theta|\sum_{i=0}^{n-2}|\alpha|^i$, thus
\begin{equation}\label{eq:normh}
\|H_{n,\alpha}\|_\infty =\|h_n\|_1= |\alpha|+
(1+|\alpha|)\left|1-|\alpha|^{n-1}\right|.
\end{equation}
It is interesting to point out that for $|\alpha|\le 1$, the above quantity is bounded from above by $1+2|\alpha|$ which is independent of $n$.

Recall from \cite{BG:book2005}  that if  $a(z)\in\mathcal W$ then $\|T(a)\|_\infty=\sum_{i\in\mathbb Z}|a_i|<\infty$, moreover, if $a(z)\in\mathcal W_S$ then $\sum_{i=1}^{\infty} |a_i|\le\frac12\|T(a)\|_\infty$. Therefore, if $|\alpha|\le 1$ then $\|P_\alpha(a)\|_\infty\le \|T(a)\|_\infty+\sup_i \|H_{i,\alpha}\|_\infty\cdot\sum_{i=1}^\infty |a_i|$ so that
\[
\|P_{\alpha}(a)\|_\infty\le \left(\frac32+|\alpha| \right)
\|T(a)\|_\infty,\quad \hbox{for } |\alpha|\le 1. 
\]
In other words, for $|\alpha|\le 1$, if  $a(z)\in\mathcal W_S$, then the matrix $P_\alpha(a)$, defined in \eqref{eq:palphainf}, has  bounded infinity norm.

If $|\alpha|>1$, in general for $a(z)\in\mathcal W_S$ it is not guaranteed that $\|P_\alpha(a)\|_\infty<\infty$, in fact, according to \eqref{eq:normh}, the norm $\|H_{n,\alpha}\|_\infty$ grows as $O(|\alpha|^n)$.
However, assuming that $a(z)\in\mathcal W_S$ is also analytic for $z$ in the annulus 
\begin{equation}\label{eq:annulus}
\mathcal A(r)=\{z\in\mathbb C~|\quad \frac 1r<|z|<r\}
\end{equation}
for some $r>|\alpha|>1$,
then it is possible to show the 
boundedness of $\|P_\alpha(a)\|_\infty$. In fact, from the analyticity of $a(z)$ in $\mathcal A(r)$ it follows that the series $\sum_{n=0}^\infty |\alpha|^n |a_n|$ is convergent, so that $\sum_{n=1}^\infty\|a_n H_{n,\alpha}\|_\infty$ is bounded from above \cite[Theorem 4.4c]{henrici}.

A similar analysis can be carried out by replacing the infinity norm with the 2-norm.
In fact, by definition, we have
$\|H_{n,\alpha}\|_2^2=\rho(H_{n,\alpha}^*H_{n,\alpha})$, where $\rho(\cdot)$ denotes the spectral radius of a matrix and the superscript ``$*$'' denotes conjugate transposition. 
Since $H_{n,\alpha}$ is possibly nonzero only in the $n\times n$ leading principal submatrix $\widehat H_{n,\alpha}$ of $H_{n,\alpha}$, we have $\|H_{n,\alpha}\|_2^2=\rho(\widehat H_{n,\alpha}^*\widehat H_{n,\alpha})$. Moreover we have $|\widehat H_{n,\alpha}^*\widehat H_{n,\alpha}|\le|\widehat H_{n,\alpha}^*|\cdot
|\widehat H_{n,\alpha}|=|\widehat H_{n,\alpha}|^2$, { where the matrix inequality is meant entry-wise}. From the theory of nonnegative matrices \cite{berman}, we know that if $A$ and $B$ are $n\times n$ matrices such that $|B|\le A$ then $\rho(B)\le \rho(A)$, thus we obtain
\[
\|H_{n,\alpha}\|_2^2\le
\rho(|\widehat H_{n,\alpha}|^2)=\rho(|\widehat H_{n,\alpha}|)^2\le
\|\widehat H_{n,\alpha}\|_\infty^2,
\]
where the latter inequality holds since $\rho(A)\le\|A\|_\infty$.
Thus, since $\|H_{n,\alpha}\|_\infty=\|\widehat H_{n,\alpha}\|_\infty$,  from \eqref{eq:normh} we get
\[
\|H_{n,\alpha}\|_2\le |\alpha|+
(1+|\alpha|)\left|1-|\alpha|^{n-1}\right|.
\]
This inequality, together with the property $\|T(a)\|_2=\max_{|z|=1}|a(z)|$, implies that $\|P_\alpha(a)\|_2<\infty$ if $|\alpha|\le 1$, or if $a(z)$ is analytic in the annulus $\mathcal A(r)$ defined in \eqref{eq:annulus}, for some $r>|\alpha|>1$.

Define the following set
\begin{equation}\label{eq:palpha}
\mathcal P_\alpha=\left\{P_\alpha(a)=\sum_{i=0}^\infty a_iP_{i,\alpha}~|\quad a(z)=a_0+\sum_{i=1}^\infty a_i(z^i+z^{-i})\in\mathcal W_S\right\}.
\end{equation}

We may synthesize the analysis of this section by the following result.

\begin{theorem}\label{th:conc}
Let $a(z)=a_0+\sum_{i=1}^\infty a_i(z^i+z^{-i})\in\mathcal W_S$, $\alpha\in\mathbb C$  and let $P_\alpha(a)$ be the matrix defined in \eqref{eq:palphainf}.
Assume that one of the following conditions is satisfied:
\begin{enumerate}
    \item $|\alpha|\le 1$;
    \item $|\alpha|> 1$ and $a(z)$ is analytic in the annulus $\mathcal A(r)$ defined in \eqref{eq:annulus},
    where $r>|\alpha|$.
\end{enumerate}
Then $\| P_\alpha(a)\|_\infty<\infty$,  $\| P_\alpha(a)\|_2<\infty$, and
\[
P_\alpha(a)=T(a)+H_\alpha(a),\quad H_\alpha(a)=\sum_{i=1}^\infty a_iH_{i,\alpha}.
\]
Moreover, if $|\alpha|\le 1$, then the set $\mathcal P_\alpha$ is a Banach algebra
 with respect to $\|\cdot\|_\infty$ and $\|\cdot\|_2$. If $|\alpha|>1$, then for any fixed $r>|\alpha|$, the set $\mathcal P_\alpha(r)=\left\{P_\alpha(a) ~|\quad a(z)\in\mathcal W_S \textrm{ analytic in}\right.$ $\left.\mathcal{A}(r)\right\}$ is a Banach algebra.
\end{theorem}
\begin{proof}
The boundedness of $\|P_\alpha(a)\|_\infty$ and $\|P_\alpha(a)\|_2$ follows from the arguments discussed before the theorem. Concerning the structure of algebra of the set $\mathcal P_\alpha$, observe that if $P_\alpha(a)=T(a)+H_\alpha(a)$ and $P_\alpha(b)=T(b)+H_\alpha(b)$ are matrices in $\mathcal P_\alpha$, so that $a(z),b(z)\in\mathcal W_S$, then $P_\alpha(a)+P_\alpha(b)=T(a)+T(b)+H_\alpha(a)+H_\alpha(b)$, and, by linearity,  $P_\alpha(a)+P_\alpha(b)=T(a+b)+H_\alpha(a+b)=P_\alpha(a+b)\in\mathcal P_\alpha$. Similarly, $P_\alpha(a)P_\alpha(b)=T(a)T(b)+H_\alpha(a)P_\alpha(b)+
T(a)H_\alpha(b)$, and in view of Theorem \ref{thm:bg} we have
\[
P_\alpha(a)P_\alpha(b)=T(ab)+H(a_-)H(b_+)+H_\alpha(a)P_\alpha(b)+
T(a)H_\alpha(b).
\]
That is, $P_\alpha(a)P_\alpha(b)=T(ab)+K$, for $K$ compact.
On the other hand, since $c(z)=a(z)b(z)\in\mathcal W_S$, then $P_\alpha(c)\in\mathcal P_\alpha$ so that $P_\alpha(c)-P_\alpha(a)P_\alpha(b)=H_\alpha(c)-K\in\mathcal P_\alpha$ is compact. { But a compact matrix in $\mathcal P_\alpha$} has null Toeplitz part, i.e., it is associated with the null symbol, thus,  it coincides with the null matrix  so that $P_\alpha(c)=P_\alpha(a)P_\alpha(b)$. This proves that $\mathcal P_\alpha$ is a matrix algebra. \qed
\end{proof}

\subsection{Computational issues}
From the proof of Theorem \ref{th:conc}, it turns out that
given $a(z),b(z),c(z)\in\mathcal W_S$ such that $c(z)=a(z)+b(z)$, then $P_\alpha(a)+P_\alpha(b)=P_\alpha(c)$. Similarly, if $c(z)=a(z)b(z)$ then $P_\alpha(a)P_\alpha(b)=P_\alpha(c)$. Finally if $a(z)\ne 0$ for $|z|=1$, then $c(z)=1/a(z)\in\mathcal W_S$ and $P_\alpha(c)P_\alpha(a)=I$. 

As a consequence of the above properties we may say that if $f(z):\mathbb C\to\mathbb C$ is a rational function defined in the set $a(\mathbb T)$, so that $f(a(z))\in\mathcal W_S$, then $f(P_\alpha(a))=P(f(a))$ if $|\alpha|\le 1$. Moreover, the same equality holds if $|\alpha|>1$ and $f(a(z))$ is analytic in $\mathcal A(r)$ for { some} $r>|\alpha|$.

These properties reduce the matrix arithmetic in the class $\mathcal P_\alpha$ to function arithmetic in $\mathcal W_S$, that can be easily implemented whatever is the value of $\alpha$ with $|\alpha|\le 1$, by following the same techniques developed in \cite{bmr:cqt} for QT arithmetic.
The same arithmetic can be implemented for $|\alpha|>1$ in the case where the symbols are analytic in $\mathcal A(r)$.

From the algorithmic point of view, we restrict the attention to the case where $a(z)=a_0+\sum_{i=1}^na_i(z^i+z^{-i})$ and $b(z)=b_0+\sum_{i=1}^mb_i(z^i+z^{-i})$ are Laurent polynomials. In the case of Laurent series we need to truncate $a(z)$ and $b(z)$ to a finite degree.
In this case, the operation of addition is almost immediate since for $c(z)=a(z)+b(z)$ we have
$c_i=a_i+b_i$, $i=0,\ldots,\max(m,n)$ where we assume that $a_i=0$ and $b_i=0$ for $i$ out of range.

The case where $c(z)=a(z)b(z)$ is slightly more complicated. Here, for a given integer $N$ we need a principal $N$-th root of 1, $\omega_N=\cos(2\pi/N)+\iu \sin(2\pi/N)$, where $\iu$ is the imaginary unit such that $\iu^2=-1$.  The Laurent polynomial $c(z)$ is such that $z^{m+n}c(z)$ is a polynomial of degree at most $2(m+n)$. In order to compute the coefficients $c_i$, it is sufficient to fix an integer $N>2(n+m)$,  and compute the values $a(\omega_N^i)$,  $i=1,\ldots,N$  with 2 FFTs, and then compute the $N$ products $y_i=a(\omega_N^i)b(\omega_N^i)$ for $i=1,\ldots, N$. The values of the coefficients $c_i$ are finally recovered by interpolating the pairs $(\omega_N^i,y_i)$ by means of IFFT.
The overall cost of this procedure is $O((m+n)\log(m+n))$ arithmetic operations.

A similar approach can be applied to compute $P_\alpha(a)^{-1}=P_\alpha(1/a(z))$ in the case where $a(z)\ne 0$ for $|z|=1$. In fact, in this case the function $1/a(z)$ is a Laurent series that can be approximated by a Laurent polynomial $c(z)$ of suitable degree. The coefficients and the degree of $c(z)$, together with the ratio $\hbox{cond}=\max_{z\in\mathbb T}|a(z)|/\min_{z\in\mathbb T}|a(z)|$, can be determined  dynamically by means of
Algorithm~\ref{alg:inv}.

\begin{algorithm}
\caption{Inverting a Laurent polynomial}\label{alg:inv}

{\sc Input:} $\epsilon>0$, $a_i$, $i=0,\ldots,n$, such that $a(z)=a_0+\sum_{i=1}^na_i(z^i+z^{-i})$

{\sc Output:} An estimate of cond$=\max_{|z|=1} |a(z)|/\min_{|z|=1}|a(z)|$, an integer $N>0$, and\\ coefficients $c_i$, $i=0,1,\ldots,N/2$ such that either $\max_i|r_i|\le\epsilon$, where $r(z)=a(z)c(z)-1$, \\
and $c(z)=\sum_{i=-N/2+1}^{N/2}c_{|i|}z^i$, or cond$>\epsilon^{-1}$

\quad Choose $N=2^k$, $k=\lceil\log_2 n\rceil$  so that $N> n$ and choose $\delta>\epsilon$

\quad {\sc while} $\delta>\epsilon$ {\sc do}

\quad\quad Set $N=2N$,  $w=(w_i)$, $w_i=\omega_N^i$, $i=1,\ldots,N$

\quad\quad Compute $y_i=a(w_i)$, $i=1,\ldots,N$

\quad\quad Interpolate $(w_i,1/y_i)$, $i=1,\ldots,N$ and get $t_i$, $i=-N/2+1,\ldots, N/2$ such that

\quad\quad\quad $t(w_i)=1/y_i$ for $i=1,\ldots,N$, where $t(z)=\sum_{i=-N/2+1}^{N/2}t_iz^i$

\quad\quad Set $r(z)=a(z)t(z)-1=:\sum_{i=-N/2-1}^{N/2}r_iz^i$ and $\delta=\max_i|r_i|$

\quad\quad Set cond=$\max_i(|y_i|)/\min_i(|y_i|)$

\quad\quad {\sc if} cond$>\epsilon^{-1}$ {\sc then} print 'ill conditioned problem' and exit

\quad {\sc end}

\quad Output $N/2$, cond, and $c_i:=t_i$, $i=0,\ldots,N/2$.
\end{algorithm}

The value of ``cond'' provides an estimate of the numerical invertibility of $a(z)$. Indeed, if cond=$\infty$ the function is not invertible and if cond$>\epsilon^{-1}$,
where $\epsilon$ is the machine precision, the problem of inverting $a(z)$ is numerically ill-conditioned. The cost of the above algorithm is $O(N\log N)$ arithmetic operations, where $N$ is the numerical degree of the Laurent series $1/a(z)$.

\section{Some special cases}

By choosing specific values for $\alpha$ we obtain very special subalgebras of the set of QT matrices having symmetric Toeplitz part.

If $\alpha=0$, the matrices $P_n$ have the following form, where we report the case for $n=4$:
\[
P_4=\begin{bmatrix}
 &&-1&\phantom{-1}0&\phantom{-1}1&\\
 &-1&&&&\phantom{-1}1\\
-1&&&&&&\phantom{-1}1\\
\,0&      &&&&&&\phantom{-1}1\\
\,1     &      &&&&&&&\ddots\\
&\,1\\
&&\ddots
\end{bmatrix},
\]
and for the matrix $P(a)$ we have 
\[
P(a)=T(a)-H((z^{-1}a(z))_+)
.
\]

If $\alpha=\pm 1$ we have
\[
P_4=\begin{bmatrix}
&&&\pm 1&1&\\
&&\pm 1&&&1\\
&\pm 1&&&&&1\\
\pm 1&      &&&&&&1\\
1     &      &&&&&&&\ddots\\
&1\\
&&\ddots
\end{bmatrix},
\]
moreover 
\[
P(a)=T(a)+H(a_\pm(z))
.
\]

\subsection{The finite case}
The analysis performed in Section \ref{sec:basic} for semi-infinite matrices can be easily applied to $m\times m$ matrices. More specifically, from Theorem \ref{thm:main} we may deduce a structural characterization of the matrix algebra generated by the powers of the $m\times m$ matrix
\[
A_{\alpha,\beta}=\begin{bmatrix}
\alpha&  1  \\
1     &  0   &1\\
      &\ddots&\ddots &\ddots\\
      &      &1      &0     &1\\
      &      &       &1     &\beta
\end{bmatrix}.
\]
This algebra has been investigated in  \cite{bozzo-phd,bozzo-difiore,ekstrom,olshevsky} in the cases where $\alpha,\beta\in\{0,1,-1\}$.

In fact, by adapting to the finite case the same argument used to prove Theorem \ref{thm:main} and its corollaries, we find that
\[
A_{\alpha,\beta}^n=\sum_{i=0}^{\lfloor\frac {n-1}2\rfloor}
\binom{n}{i}
P_{n-2i}^{(\alpha,\beta)}+\varphi_n I, \quad n=1,\ldots,m-1,
\]
with $P_0^{(\alpha,\beta)}=I$ and
\[
P_i^{(\alpha,\beta)}=T_m(z^i+z^{-i})+H_i^{(\alpha)}+J_mH_i^{(\beta)}J_m,\quad i=1,\ldots,m-1,
\]
where $T_m(z^i+z^{-i})$ is the $m\times m$ leading principal submatrix of $T(z^i+z^{-i})$,
$H_i^{(\alpha)}$ is the $m\times m$ leading principal submatrix of the semi-infinite Hankel matrix $H_{i,\alpha}$ and $J_m$ is the $m\times m$ permutation matrix with ones in the anti-diagonal. An example of the $m\times m$ matrix $P_i^{(\alpha,\beta)}$ is given for $i=4$ and $m=8$:
\[
P_4^{(\alpha,\beta)}=
\begin{bmatrix}
\theta\alpha^2&\theta\alpha&\theta&\alpha&1\\
\theta\alpha&\theta&\alpha&&&1\\
\theta&\alpha&&&&&1\\
\alpha& &&&&&&1\\
1& & & & & & &\beta\\
 &1& & & & &\beta&\nu\\
 & &1& & &\beta&\nu&\nu\beta\\
 & & &1&\beta&\nu&\nu\beta&\nu\beta^2
\end{bmatrix}
\]
where $\theta=\alpha^2-1$, $\nu=\beta^2-1$.

For $\alpha=\beta=0$ we obtain the $\tau$ algebra analyzed in \cite{bini-capovani}, where the matrices are simultaneously diagonalized by the discrete sine transform of kind 1. For $\alpha,\beta\in\{-1,0,1\}$  the corresponding matrices are simultaneously diagonalized by different types of discrete sine and discrete cosine transforms,  see \cite{olshevsky} for more details. These classes have been used to precondition positive definite Toeplitz systems.

For any value of $\alpha$ and $\beta$, the matrices in the algebras generated this way can be written as the sum of a Toeplitz and a Hankel matrix.

\section{QT matrix representation}\label{sec:represent}

The implementation of QT matrix arithmetic given in the package {\tt CQT-Toolbox} of \cite{bmr:cqt} relies on the fact that, given two QT-matrices
$A=T(a)+E_A$ and $B=T(b)+E_B$ then $C=A+B=T(c)+E_C$, where $c(z)=a(z)+b(z)$ and $E_C=E_A+E_B$ is still compact. In fact, since the matrices are finitely 
representable, then $E_A$ and $E_B$ have finite rank so that also rank$(E_C)$ is 
finite since $\hbox{rank}(E_C)\le\hbox{rank}(E_A)+\hbox{rank}(E_B)$. 
Similarly, relying on Theorem \ref{thm:bg}, the product $C=AB$, is written as
\[
C=T(c)+E_AT(b)+T(a)E_B+E_AE_B+H(a_-)H(b_+)=: T(c)+E_C,
\]
where $E_C=E_AT(b)+T(a)E_B+E_AE_B+H(a_-)H(b_+)$ still has finite rank that in principle can grow large, in fact 
\begin{equation}\label{eq:rkt}
\hbox{rank}(E_C)\le \hbox{rank}(E_A)+\hbox{rank}(E_B)+\hbox{rank}(H(a_-)H(a_+)).
\end{equation}
For this reason, a strategy of compression and approximate representation is introduced in \cite{bmr:cqt} in order to reduce the growth of the rank when many arithmetic matrix operations are performed. This compression strategy is the most time consuming part of the software {\tt CQT-Toolbox} \cite{bmr:cqt}. 

Assuming that the decomposition $E_A=U_AV_A^T$, 
$E_B=U_BV_B^T$ and $H(a_-)H(b_+)=U_HV_H^T$ are available, where matrices $U_A,V_A$ are formed by $k_A$ columns, $U_B,V_B$ are formed by $k_B$ columns, and
$U_H,V_H$ are formed by $k_H$ columns,  
the compression strategy requires the computation of the numerical rank
of the matrix $[U_A,T(a)U_B,U_H][V_A+T(b)^TV_B,V_B,V_H]^T$ that is formed by $k_A+k_B+k_H$ columns. In order to perform the compression, two matrices $U_C$, $V_C$ having less than
$k_A+k_B+k_H$ columns are computed such that $E_C$ is approximatively given by $U_CV_C^T$, up to an error of the order of the machine precision.

This computation, usually performed by means of SVD,
is time consuming especially in the cases where the numerical size of the compact correction gets large. This happens, say, when the bandwidth of the Toeplitz part grows large in any algorithm using the package {\tt CQT-Toolbox}.

Concerning matrix inversion, in \cite{bmr:cqt} the inverse of $A=T(a)+E_A=T(a)(I+T(a)^{-1}E_A)$ is computed through the identity \[
A^{-1}=(I+T(a)^{-1}E_A)^{-1}T(a)^{-1},
\]
where the inverse of $I+T(a)^{-1}E_A$ is computed by means of the Sherman-Woodbury-Morrison (SWM) formula, and the inverse of the Toeplitz matrix $T(a)$ is computed  relying either on the cylic reduction algorithm, or on FFT. For more details we refer the reader to \cite{bmr:cqt}.

For QT matrices $A=T(a)+E_A\in\mathcal{QT_S}$, where $a(z)\in\mathcal W_S$, the results of Section~\ref{sec:basic}
suggest us to represent 
$A$ 
in the form 
\begin{equation}\label{eq:repr}
A=P_\alpha(a)+K_A,
\end{equation}
where
$P_\alpha(a)=T(a)+H(a)\in\mathcal P_\alpha$, and $K_A=E_A-H(a)$ is compact. 

In fact in this case we find that, for $A=P_\alpha(a)+K_A$, $B=P_\alpha(b)+K_B$, we have
\[
AB=P_\alpha(c)+K_C,\quad K_C=P_\alpha(a)K_B+K_AP_\alpha(b)+K_AK_B,
\]
and $c(z)=a(z)b(z)$. Observe that, with this representation, we simply have 
\begin{equation}\label{eq:rkp}
\hbox{rank}(K_C)\le\hbox{rank}(K_A)+\hbox{rank}(K_B),
\end{equation}
and the Hankel part $H(a_-)H(b_+)$ disappears since it is included in $P_\alpha(c)$. This way, the compression of the correction requires the computation of the numerical rank of the matrix $[\hat U_A,P_\alpha(a)\hat U_B]  [P_\alpha(b)\hat V_A + \hat V_B (\hat U_B^T \hat V_A),\hat V_B]^T $, where $K_A=\hat U_A \hat V_A^T$ and 
$K_B=\hat U_B \hat V_B^T$.

A similar advantage is encountered in the computation of $A(a)^{-1}=(I+P_\alpha(a^{-1})K_A)^{-1}P_\alpha(a^{-1})$. In fact, in this case, the computation of $a(z)^{-1}$, performed by means of Algorithm 1, is much easier than inverting $T(a)$ and the major computational effort is just applying the  SWM formula to the matrix $I+P_\alpha(a^{-1})K_A$.

In general, if $f(x)$ is a rational function defined on the set $a(\mathbb T)$, then $f(A)=P_\alpha(f(a))+K_f$, where $K_f$ is compact. Moreover, in a straight-line program whose generic step is of the kind $S_m=S_i \circ S_j$ where $\circ\in\{+, =, *, / \}$ is any of the four operations,  $i,j<m$ and $S_\ell=P_\alpha(s_\ell)+K_\ell\in\mathcal{QT_S}$, one finds that $S_m=P_\alpha(s_m)+K_m$, with $s_m(z)=s_i(z)\circ s_j(z)$.
Moreover, the compact correction $K_m$ is related to the compact corrections $K_i$ and $K_j$ of the operands and to the specific operation $\circ$ used at the $m$-th step. 

Instead, by using the representation $S_m=T(s_m)+E_m$ given in \cite{bmr:cqt}, one finds that $S_m=T(s_i\circ s_j)+E_m$, where this time $E_m$ is related not only to the corrections $E_i$ and $E_j$ of the operands, but also depends on the correction 
$T(s_i)\circ T(s_j)-T(s_i\circ s_j)$ that, in the case of multiplication or inversion, can have large rank.

In other words, one expects that the matrices to be compressed in the representation of \cite{bmr:cqt} have rank larger than the rank of the matrices to be compress in the new representation through  $\mathcal P_\alpha$. 

The actual advantages of the new representation will be shown in the next section concerning the numerical experiments.

\section{Numerical experiments}
We say that $A=P_\alpha(a)+K_A\in\mathcal{QT_S}$ is finitely representable if the symbol $a(z)$ is a Laurent polynomial and the correction $K_A$ has a finite support, that is, its entries are zero outside its leading principal $m\times n$ submatrix for suitable finite values of $m$ and $n$.

We have implemented the representation of a matrix in $\mathcal{QT_S}$ in terms of the pair $(P_\alpha(a),K_A)$ together with the basic arithmetic operations. 

In this section, we report the results of 
a preliminary experimentation 
that we have performed in order to compare the efficiency of the above representation with the standard representation of general QT matrices given in the
{
{\tt CQT-Toolbox} of \cite{bmr:cqt},  see
\url{ https://github.com/numpi/cqt-toolbox}.
The experiments have been performed in Matlab v. 9.5.0.1033004 (R2018b) on a laptop with an Intel I3 processor and 16GB RAM. The functions and the scripts that execute the tests can be downloaded from \url{https://numpi.dm.unipi.it/scientific-computing-libraries/sqt/}
}

The first test concerns the solution of the quadratic matrix equation 
\begin{equation}\label{eq:qme}
AX^2+BX+C=X
\end{equation}
in the case where $A=T(a)+E_A$, $B=T(b)+E_B$, $C=T(c)+E_C$ are nonnegative matrices in $\mathcal{QT_S}$ such that $A+B+C$ is stochastic.
In this case, the solution of interest is the minimal nonnegative solution $G$
that exists and is a QT matrix, i.e., $G=T(g)+E_G$ \cite{bmm:simax}. This solution
can be computed
by means of the classical functional iterations  
\cite{blm:book}, \cite{meini:functit},
\begin{equation}\label{eq:functits}
\begin{split}
    &X_{k+1}=AX_k^2+BX_k+C,\\
    &X_{k+1}=(I-B)^{-1}(AX_k^2+C),\\
    &X_{k+1}=(I-AX_k-B)^{-1}C,~~k=0,1,\ldots,
\end{split}
\end{equation}
known as {\em Natural}, {\em Traditional}, and {\em U-based} iteration, respectively, initialized with $X_0=0$. 

We have performed experiments by representing the matrix coefficients relying both on the standard QT-matrix representation given in \cite{bmr:cqt}, and on the representation \eqref{eq:repr} provided in Section \ref{sec:represent}. 

Let 
{ $X_{k+1}-X_k=T(\sigma_k)+E_k$} 
be the standard representation of the QT matrix $X_{k+1}-X_k$. Similarly, let
{$X_{k+1}-X_k=P_\alpha(\sigma_k)+K_k$}
be the new representation. Denote $\epsilon_k$  the maximum modulus of the coefficients of the Laurent series $\sigma_k(z)$. The iterations have been halted if 
{$\max(\epsilon_k,\|E_k\|_\infty)<5\cdot 10^{-15}$}, 
and if 
{ $\max(\epsilon_k,\|K_k\|_\infty)<5\cdot 10^{-15}$}
for the standard and for the new representation, respectively.

We have chosen the following values for the symbol of the Toeplitz part of $A$, $B$, and $C$:
{ $a(z)=\frac1{100}(10z^{-1}+10+10z)$}, $b(z)=\frac1{100}(8z^{-1}+23+8z)$, $c(z)=\frac1{100}(10z^{-1}+11+10z)$, respectively. While the compact corrections have been chosen as {$E_A=\frac{10}{100}E$}, $E_B=\frac8{100}E$, $E_C=\frac{10}{100}E$, where $E=e_1e_1^T$. With these values, $A+B+C$ is stochastic and the Markov chain associated with the random walk in the quarter plane defined by $A,B$ and $C$ is pretty close to be null recurrent so that the functional iterations \eqref{eq:functits} take many steps to arrive at numerical convergence.

Observe that, since for $v(z)=v_{-1}z^{-1}+v_0+v_1 z$ we have $T(v)=P_0(v)$,  the matrices $A,B,C$ can be viewed  as $A=P_0(a)+E_A$, $B=P_0(b)+E_B$, $C=P_0(c)+E_C$. On the other hand, 
we also have $A=P_1(a)$, $B=P_1(b)$, $C=P_1(c)$, that is, $A,B,C\in\mathcal P_1$. Therefore, since $\mathcal P_1$ is an algebra and $A,B,C\in\mathcal P_1$, then the sequences $X_k$ generated by the three iterations \eqref{eq:functits} belong to $\mathcal P_1$.
 This way, $X_k$ can be written as $X_k=P_1(x_k)$, where $x_k(z)\in\mathcal{W_S}$ satisfies one of the iterations
 \[
 \begin{split}
     &x_{k+1}(z)=a(z)x_k(z)^2+b(z)x_k(z)+c(z),\\
     &x_{k+1}(z)=(1-b(z))^{-1}(a(z)x_k(z)^2+c(z)),\\
     &x_{k+1}(z)=(1-a(z)x_k(z)-b(z))^{-1}c(z),~~k=0,1,\ldots
 \end{split}
 \]
 according to the corresponding iteration in \eqref{eq:functits}, respectively, with $x_0(z)=0$.
 Moreover, since $\lim_k\|x_k(z)-g(z)\|_W=0$ (see \cite{bmm:simax}), where, for any $|z|=1$, $g(z)$ is the solution of minimum modulus of the equation $a(z)x(z)^2+b(z)x(z)+c(z)=x(z)$ (see \cite{bmmr:scicomp}),
 we have $G=P_1(g)$.

If we use the representation  $A=P_0(a)+E_A$, $B=P_0(b)+E_B$, $C=P_0(c)+E_C$, then for the sequences $X_k$ generated by \eqref{eq:functits}, we have $X_k=P_0(x_k)+K_k$, where $K_k$ is compact. Moreover, $G=\lim_k X_k=P_0(g)+K$, and $K=\lim_k K_k$, where convergence is meant in the infinity norm. 

In Table \ref{tab:qme}, for each fixed-point method { in \eqref{eq:functits}} denoted as 1,2, and 3, we report the CPU time, the number of iterations, the length of the symbol $g(z)$, the size and the rank of the compact correction in $G$ together with  { residual error $\|AX^2+BX+C-X\|_\infty$, where $X$ is the computed solution. }
{ We also report, between parentheses and in bold, the CPU time obtained when the representation $X_i=P_1(x_i)$ is used. In this case, the operations involving the compact correction are avoided and the CPU time is smaller.}

\begin{table}
\scriptsize
    \centering
    \begin{tabular}{c|ccc|ccc}
   &\multicolumn{3}{c|}{$\mathcal{QT_S}$}&    \multicolumn{3}{c}{$\mathcal{QT}$}\\\hline
  Algorithm&1&2& 3&
         1&2&3\\\hline
    CPU   &40.3 ({\bf 0.77})&29.7 ({\bf 0.54})&47.6 ({\bf 0.71})&125.1&172.9&88.3    \\
    Iterations &2007&1289&719&2124&1333&700   \\
    Symbol size &1203&1204&1215&1086&1095&1070\\
    Correction size &1401&1407&1405&1206&1201&1446\\
    Correction rank &23&23&23&19&19&19\\ 
    Error &5.9E-15&2.9E-15&1.4E-15 &2.0E-15 &6.0E-16 &6.6E-16
    \end{tabular}
    \caption{\footnotesize Solution of the quadratic matrix equation \eqref{eq:qme} by means of Algorithm 1, 2, and 3 of \eqref{eq:functits} where the coefficients are QT matrices in $\mathcal P_1$, represented either as entries in $\mathcal P_1$ (CPU time between parentheses and in bold), or in the form $W+K$ for $W\in\mathcal P_0$ and $K$ compact (column $\mathcal{QT_S}$), or in the form $T+E$, where $T$ is Toeplitz and $E$ is compact (column $\mathcal{QT}$). Besides the CPU time in seconds, the table reports the number of iterations, the number of coefficients in the symbol, the size and the rank of the correction, { together with the infinity norm of the residual error $\|AG^2+BG+C-G\|_\infty$, where $G$ is the computed solution}. 
    }
    \label{tab:qme}
\end{table}

We may see how the new representation provides a speed-up in the CPU time that ranges roughly from 1.6 to  5.8. Moreover, comparing with the representation of $X_i$ as matrix in $\mathcal P_1$, we have speedups in the range $[124, 320]$.
Observe also that the residual errors are close to the machine precision.

We may observe a slight discrepancy between the number of iterations required for the two different representations. This is due to the fact that the stop condition involves the norm of the compact correction that can be different in the two representations. 
{ It must be said that in the experiments involving the {\tt CQT-toolbox}, cpu time, number of iterations, and errors can slightly change in different executions of the code. The reason is that the compression operation performed in the {\tt CQT-toolbox} is randomized.}

The second test concerns the computation of the square root of a matrix $A$ performed by means of the algorithm  \cite[Equation (6.20)]{higham:book} designed in \cite{iannazzo}.
We have considered the class of matrices given by $A_\gamma=T(a)$, where { $T(a)$ is the symmetric Toeplitz matrix associated with the symbol 
$a(z)=\gamma+4(z+z^{-1})+3(z^2+z^{-2})+2(z^3+z^{-3})+(z^4+z^{-4})]$}, for $\gamma>5$. This matrix is positive definite for $\gamma>5$, while if $\gamma=5$ it is not invertible. 

Table \ref{tab:test1} reports the CPU time in seconds, the number of coefficients of the symbol, the size and the rank of the correction,  together with the residual error $\|X^2-A\|_\infty$, where $X$ is the computed matrix, for values of $\gamma =5+\delta$ with $\delta=10^{-1}$, $10^{-2}$, $10^{-3}$. Between parentheses the CPU time required for computing the symbol is reported.
This time corresponds to computing the square root of any matrix of the kind $P_\alpha(a)$.

Observe that also in this case, the speed up with respect to the QT representation ranges from 2.22 to 3.15. While comparing the CPU time for computing the symbol, i.e., for computing the square root of $A=P_\alpha(a)$,  with the time required by the QT representation we get  speed-ups in the range $[78,233]$.

In all the experiments that we have performed, the residual errors obtained with the new representation are comparable with ones obtained with the {\tt CQT-toolbox}.

\begin{table}
\scriptsize
    \centering
    \begin{tabular}{c|ccc|ccc}
    &\multicolumn{3}{c|}{$\mathcal{QT_S}$}&    \multicolumn{3}{c}{$\mathcal{QT}$}\\\hline
    $\delta$&$10^{-1}$&$10^{-2}$& $10^{-3}$&
      $10^{-1}$&$10^{-2}$& $10^{-3}$\\\hline
      
    CPU   &0.35 ({\bf 0.01})&1.32 ({\bf  0.02})&7.93 ({\bf 0.1}) &0.78&4.16&23.3    \\
    
    Iterations &7&8&9&7&8&9   \\
    
    Symbol size &352&1014&2911&317&957&2922\\
    
    Correction size &417&1443&4283& 338&1023&3175\\
    
    Correction rank &45&60&74 &3&3&2\\ 
    
    Residual error &1.0E-14&1.5E-14&1.9E-14 &6.7E-13 &9.5E-13 &1.2E-12
    \end{tabular}
        \caption{\footnotesize Computation of the square root of the SQT matrix $A=T(a)$, where $a(z)=5+\delta+ 4(z+z^{-1})+3(z^{-2}+z^2)+2(z^{-3}+z^3) +(z^{-4}+z^4)$, using the standard ($\mathcal{QT}$) and the  new ($\mathcal{QT_S}$) representation of $A$. Values of the CPU time in seconds, number of iterations, length of the symbol, size and rank of the compact correction, together with the residual error $\|X^2-A\|_\infty$, are reported. Between parentheses, in bold, the CPU time required for $A=P_1(a)$.}
    \label{tab:test1}
\end{table}

\section{Conclusions}
A class, depending on a parameter $\alpha$, of subalgebras of semi-infinite quasi-Toeplitz  matrices having a symmetric symbol has been introduced.
A suitable basis of these subalgebras, as linear spaces, has been determined and exploited to represent in a more efficient way quasi-Toeplitz matrices associated with a symmetric symbol. 
A preliminary experimentation has been performed concerning the computation of the matrix square root and  the solution of a quadratic matrix equation. From our tests it turns out that the new representation is more effective  with respect to the one given in the Matlab toolbox of \cite{bmr:cqt}. 
A structured analysis of these algebras has been carried out both for semi-infinite matrices and for finite matrices where the algebras depend on two parameters $\alpha$ and $\beta$. The latter class includes the algebras investigated in \cite{bini-capovani,bozzo-phd,bozzo-difiore,ekstrom,olshevsky} obtained for $\alpha,\beta\in\{-1,0,1\}$.
To our knowledge, the case of general values of $\alpha$ and $\beta$ has not been analyzed in the literature.



\section*{Funding}
MUR Excellence Department Project awarded to the Department of Mathematics, University of Pisa, CUP I57G22000700001. 

European Union - NextGenerationEU under the National Recovery and Resilience Plan (PNRR) - Mission 4 Education and research - Component 2 From research to business - Investment 1.1 Notice Prin 2022 - DD N. 104  2/2/2022, titled Low-rank Structures and Numerical Methods in Matrix and Tensor Computations and their Application, proposal code 20227PCCKZ – CUP 
I53D23002280006.

Spoke 1 ``FutureHPC \& BigData'' of the Italian Research Center on High-Performance Computing, Big Data and Quantum Computing
(ICSC) funded by MUR Missione 4 Componente 2 Investimento 1.4: Potenziamento strutture di ricerca e creazione di ``campioni nazionali di R\&S(M4C2-19)'' - Next Generation EU (NGEU).

GNCS of INdAM. The second author is member of GNCS of INdAM.





\end{document}